\documentclass[11pt]{amsart}

\usepackage{amsfonts}
\usepackage{amssymb}
\usepackage{hyperref}
\usepackage{stmaryrd}

\usepackage{tikz}
\usetikzlibrary{matrix,arrows,decorations.pathmorphing}

\theoremstyle{plain}
\newtheorem{theorem}{Theorem}[section]
\newtheorem{lem}[theorem]{Lemma}
\newtheorem{cor}[theorem]{Corollary}
\newtheorem{prop}[theorem]{Proposition}

\theoremstyle{definition}
\newtheorem{defi}[theorem]{Definition}
      
\theoremstyle{remark}

\newcommand{\lemref}[1]{\hyperref[#1]{Lemma \ref*{#1}}}
\newcommand{\thmref}[1]{\hyperref[#1]{Theorem \ref*{#1}}}
\newcommand{\propref}[1]{\hyperref[#1]{Proposition \ref*{#1}}}
\newcommand{\corref}[1]{\hyperref[#1]{Corollary \ref*{#1}}}
\newcommand{\defref}[1]{\hyperref[#1]{Definition \ref*{#1}}}
\newcommand{\remref}[1]{\hyperref[#1]{Remark \ref*{#1}}}
\newcommand{\conjref}[1]{\hyperref[#1]{Conjecture \ref*{#1}}}

\newcommand{\llb}{\llbracket}
\newcommand{\rrb}{\rrbracket}

\catcode`\@=11
\newdimen\cdsep
\def\cdstrut{\vrule height .6\cdsep width 0pt depth .4\cdsep}
\def\@cdstrut{{\advance\cdsep by 2em\cdstrut}}

\def\arrow#1#2{
  \ifx d#1
    \llap{$\scriptstyle#2$}\left\downarrow\cdstrut\right.\@cdstrut\fi
  \ifx u#1
    \llap{$\scriptstyle#2$}\left\uparrow\cdstrut\right.\@cdstrut\fi
  \ifx r#1
    \mathop{\hbox to \cdsep{\rightarrowfill}}\limits^{#2}\fi
  \ifx l#1
    \mathop{\hbox to \cdsep{\leftarrowfill}}\limits^{#2}\fi
}
\catcode`\@=12

\def \F {\mathbb{F}}
\makeatletter
\newcommand*{\defeq}{\mathrel{\rlap{%
                     \raisebox{0.27ex}{$\m@th\cdot$}}%
                     \raisebox{-0.27ex}{$\m@th\cdot$}}%
                     =}

\numberwithin{equation}{section}
\makeatother

\makeatletter
\def\@setcopyright{}
\def\serieslogo@{}
\makeatother

\input xy
\xyoption{all} 
 
\begin{document}

\title{The Hanna Neumann conjecture for Demushkin Groups}

\author{Andrei Jaikin-Zapirain}
\address{Departamento de Matem\'{a}ticas, Universidad Aut\'{o}noma de Madrid and Instituto de Ciencias Matem\'{a}ticas, CSIC-UAM-UC3M-UCM}
\email{andrei.jaikin@uam.es}

\author{Mark Shusterman}
\address{Raymond and Beverly Sackler School of Mathematical Sciences, Tel-Aviv University, Tel-Aviv, Israel}
\email{markshus@mail.tau.ac.il}

\date{}
\begin{abstract}

We confirm the Hanna Neumann conjecture for topologically finitely generated closed subgroups $U$ and $W$ of a nonsolvable Demushkin group $G$. Namely, we show that
\begin{equation*}
\sum_{g \in U \backslash G/W} \bar d(U \cap gWg^{-1}) \leq \bar d(U) \bar d(W)
\end{equation*}
where $\bar d(K) = \max\{d(K) - 1, 0\}$ and $d(K)$ is the least cardinality of a topological generating set for the group $K$.

\end{abstract}

\maketitle

\section{Introduction}

Howson has shown in \cite{Hows} that the intersection of two finitely generated subgroups $U,W$ of a free group $F$ is finitely generated.
The problem of obtaining the optimal bound on the number of generators of the intersection has been posed by Hanna Neumann in \cite{Nh}. 
She conjectured that
\begin{equation}
\bar d(U \cap W) \leq \bar d(U) \bar d(W).
\end{equation}
A lot of works on the conjecture followed, and in particular, Walter Neumann conjectured in \cite{Nw} that the strengthened inequality
\begin{equation}
\sum_{g \in U \backslash F/W} \bar d(U \cap gWg^{-1}) \leq \bar d(U) \bar d(W)
\end{equation}
holds.
This strengthened conjecture motivated a long line of works that culminated in solutions by Friedman in \cite{F} and Mineyev in \cite{M}.

The pro-$p$ analog of the Hanna Neumann conjecture has a similar timeline.
Howson's theorem for free pro-$p$ groups has been established by Lubotzky in \cite{Lub},
and the strengthened Hanna Neumann conjecture for these groups has been obtained in \cite{J} by Jaikin-Zapirain, whose arguments led to a new proof of the original strengthened Hanna Neumann conjecture.

In this work we focus on Demushkin groups (pro-$p$ Poincar\'{e} duality groups of dimension $2$). These are finitely generated one-relator pro-$p$ groups $G$ for which the cup product
\begin{equation}
\cup \colon H^1(G,\mathbb{F}_p) \times H^1(G,\mathbb{F}_p) \to H^2(G,\mathbb{F}_p)
\end{equation}
is non-degenerate.
Demushkin groups appear in arithmetic algebraic geometry as maximal pro-$p$ quotients of \'{e}tale fundamental groups,
in combinatorial group theory as pro-$p$ completions of surface groups, and in number theory as Galois groups of maximal $p$-extensions of $p$-adic fields.
This number theoretic appearance (and its variants) are responsible for the attention payed to the properties of Demushkin groups both in classical textbooks on Galois cohomology such as \cite{NSW, Ser} and in modern research works in Galois theory such as \cite{BLMS, B0, B, Ef2, LLMS, MW0, MW, Wing, Y}.

Demushkin groups were also studied for their own sake, for instance, in \cite{SerD} by Serre and in \cite{Lab0, Lab} by Labute.
Their group theoretic properties continue to attract attention as can be seen from \cite{DL, Koc, KS, KZ0, KZ1, Sh, ShZ, SZ, Son}.
In particular, Howson's theorem for these groups has been obtained  by Shusterman and Zalesskii in \cite{ShZ}.
It is therefore very natural to ask whether the Hanna Neumann conjecture is true also for Demushkin groups.
We answer the strengthened form of this question in the affirmative.

\begin{theorem} \label{FirstRes}

Let $G$ be a nonsolvable Demushkin group, let $U$ and $W$ be two closed topologically finitely generated nontrivial subgroups of $G$, and set
\begin{equation}
S \defeq \{g \in U \backslash G/ W \ | \ U \cap gWg^{-1} \neq 1 \}.
\end{equation}
Then $S$ is finite and
\begin{equation} \label{HNG}
\sum_{g \in S} \big( d(U \cap gWg^{-1}) - 1 \big) \leq \big(d(U)-1\big) \big(d(W)-1\big). 
\end{equation}

\end{theorem}

Note that since Demushkin groups contain free pro-$p$ groups, this theorem extends Jaikin-Zapirain's result from \cite{J}.
Our assumption that $G$ is nonsolvable is necessary, since otherwise one can take $U = W \cong \mathbb{Z}_3$
in $G \defeq U \rtimes \mathbb{Z}_3$ and, in this case, $S$ is infinite. 
If we take two open subgroups in $G$, then $S$ is finite but \eqref{HNG} may fail to hold.
By Labute's classification from \cite{Lab}, the nonsolvability of $G$ is tantamount to $d(G) > 2$ (but we shall not use this fact).

As already noted in \cite{Nw}, the case where either $U$ or $W$ is open in $G$ reduces to a simple calculation (using \eqref{RankFor} in our case), so we shall assume throughout that the indices $[G: U]$ and $[G:W]$ are infinite.

Even though the possibility of extending the Hanna Neumann conjecture to (discrete) surface groups has already been considered in \cite{Som},
\thmref{FirstRes} is the first extension of the conjecture to groups that are not free.

Our proof builds on ideas from the aforementioned work \cite{J} of Jaikin-Zapirain.
As in \cite{J} (and in other proofs of the Hanna Neumann conjecture) we introduce some (analog of an) $L^2$-invariant,
and reduce the conjecture to a certain submultiplicativity property of this invariant.
A difficulty then arises as the arguments of \cite{J} are based on the fact that $\mathbb{F}_p \llb G/W\rrb$ is a virtually free $G$-module, once $G$ is a free pro-$p$ group. Even the finiteness of $S$ does not immediately carry over to the Demushkin case.

As a first substitute for virtual freeness, we generalize the arguments from the proof of Howson's theorem for Demushkin groups (by Shusterman and Zalesskii), deducing the finiteness of $S$ by a tricky reduction to the free pro-$p$ case.
The second substitute is that $\mathbb{F}_p \llb G/W\rrb$ is virtually a one-relator $G$-module, once $G$ is a Demushkin group.
In order to show that our `$L^2$-invariant' vanishes on one-relator modules, 
and for other key arguments in the proof (in the spirit of \cite{J}), we need to establish (an analogue of) the Atiyah conjecture for Demushkin groups (see Section 5.2).

For free pro-$p$ groups,
the Atiyah conjecture is deduced in \cite{J} from the fact that the consecutive quotients in the descending central series are torsion-free.
By \cite{KKL}, this does not generalize to Demushkin groups.
As a replacement, we show that any pro-$p$ Demushkin group is an inverse limit of groups obtained from copies of $\mathbb{Z}_p$ by semi-direct products.
The Atiyah conjecture is deduced from that.
Along the way we also obtain the following.

\begin{theorem}

The Kaplansky zero-divisor conjecture over $\mathbb{F}_p$ is true for any (torsion-free) pro-$p$ Demushkin group $G$. Namely, the completed group algebra $\F_p \llb G \rrb$ has no non-trivial zero-divisors.

\end{theorem}

\section{Preliminaries}

\subsection{Homology}

We fix once and for all a prime number $p$.
For a finitely generated pro-$p$ group $G$, the completed group algebra of $G$ over $\F_p$ is
$$\F_p \llb G \rrb = \varprojlim_{N \lhd_o G} \F_p[G/N].$$
We consider the category of (left profinite) $\F_p \llb G \rrb$-modules.
Let $M$ be such a module, and note that $M$ is finitely generated if and only if its maximal $G$-trivial quotient $M_G$
(which can also be identified with the homology group $H_0(G,M)$) is of finite dimension over $\F_p$.
We say that $M$ is finitely related if
\begin{equation}
\dim_{\F_p} H_1(G, M) < \infty.
\end{equation}
If $M$ is also finitely generated, we say that $M$ is finitely presented.
Equivalently, $M$ fits into an exact sequence of $\F_p \llb G \rrb$-modules
\begin{equation}
0 \to K \to \F_p \llb G \rrb^d \to M \to 0
\end{equation}
where $d \in \mathbb{N}$ and $K$ is finitely generated.
For example, the (trivial) one-dimensional $G$-module satisfies
\begin{equation}
\dim_{\F_p} H_1(G, \F_p) = d(G) < \infty.
\end{equation}

We will make free use of the homological long exact sequence associated to a short exact sequence of $\F_p \llb G \rrb$-modules by \cite[Proposition 6.1.9]{RZ}. An example is the following.
\begin{prop} \label{HFPprop}

Suppose that $G$ is finitely presented, let $M$ be a finitely presented $\F_p \llb G \rrb$-module, 
and let $M_0$ be an open $G$-submodule of $M$.
Then $M_0$ is also finitely presented. 

\end{prop}

\begin{proof}

As the only simple $G$-module is $\F_p$, we may assume (by an inductive argument) that the inclusion of $M_0$ in $M$ is encoded in the exact sequence
\begin{equation}
0 \to M_0 \to M \to \F_p \to 0.
\end{equation}
The associated long exact sequence provides us with the inequality
\begin{equation}
\dim_{\F_p} H_0(G, M_0) \leq \dim_{\F_p} H_1(G, \F_p) + \dim_{\F_p} H_0(G, M)  
\end{equation}
where the right hand side is finite since $G$ is a finitely generated pro-$p$ group and $M$ is a finitely generated $G$-module.
We conclude that $M_0$ is a finitely generated $G$-module.
The aforementioned long exact sequence also gives
\begin{equation}
\dim_{\F_p} H_1(G, M_0) \leq \dim_{\F_p} H_2(G, \F_p) + \dim_{\F_p} H_1(G, M)  
\end{equation}
where now the right hand side is finite since $G$ is finitely presented and $M$ is finitely related.
It follows that $M_0$ is finitely presented.
\end{proof}

If $H$ is a (closed) subgroup of $G$ and $M$ is an $\F_p \llb H \rrb$-module,
we can induce $M$ from $H$ to $G$ by
\begin{equation}
\mathrm{Ind}_H^G M \defeq  \F_p \llb G \rrb \ \widehat\otimes_{\F_p \llb H \rrb} \ M
\end{equation} 
obtaining an $\F_p \llb G \rrb$-module.
Induction is an exact functor, (naturally) satisfying the transitivity formula
\begin{equation}
\mathrm{Ind}_K^G \mathrm{Ind}_H^K M \cong_G \mathrm{Ind}_H^G M
\end{equation}
for every subgroup $K$ of $G$ that contains $H$.
We often use Shapiro's lemma (see \cite[Theorem 6.10.8 (d)]{RZ}) saying that we have (natural) isomorphisms
\begin{equation}
H_*(G, \mathrm{Ind}_H^G M) \cong H_*(H,M).
\end{equation}

Furthermore, if $M$ is a $G$-module, we have the (natural) isomorphism 
\begin{equation} \label{SomeIndEq}
\mathrm{Ind}_H^G M \cong M \ \widehat\otimes_{\F_p} \ \F_p \llb G/H \rrb
\end{equation}
of $G$-modules. 
If moreover $H$ is open in $G$ then the $G$-module $\F_p[G/H]$ admits a filtration of length $[G : H]$ with one-dimensional ($\cong \F_p$) consecutive quotients.
As a result, the induced module
\begin{equation} 
M \ \widehat\otimes_{\F_p} \ \F_p[G/H]
\end{equation}
from equation \eqref{SomeIndEq} admits a filtration (by $G$-submodules) of length $[G : H]$ with consecutive quotients isomorphic to
\begin{equation}
M \ \widehat\otimes_{\F_p} \ \F_p \cong M.
\end{equation}

Let us examine another example of a (possibly) finitely related module.
For that pick subgroups $U,W$ of $G$, and consider $\F_p \llb G/W \rrb$ as a $U$-module.
Using Melnikov's direct sum (over a profinite set) and Mackey's formula,
one can write an isomorphism of $U$-modules
\begin{equation} \label{UDecIs}
\F_p \llb G/W \rrb  \cong \bigoplus_{g \in U \backslash G/W} \F_p \llb U/ U \cap gWg^{-1} \rrb.
\end{equation}
It is shown in \cite[Lemma 3.3]{Mel} that homology commutes with (profinite) direct sums, so we have
\begin{equation} \label{Eq0}
H_* \big( U, \F_p \llb G/W \rrb \big) \cong 
\bigoplus_{g \in U \backslash G/W}  H_* \big(U, \F_p \llb U/ U \cap gWg^{-1} \rrb \big).
\end{equation}
Applying Shapiro's lemma to the right hand side gives
\begin{equation} \label{Eq2}
H_* \big( U, \F_p \llb G/W \rrb \big) \cong  \bigoplus_{g \in U \backslash G/W}  H_* \big(U \cap gWg^{-1}, \F_p \big)
\end{equation}
so for the first homology we get that
\begin{equation} \label{Eq3}
\dim_{\F_p} H_1 \big( U, \F_p \llb G/W \rrb \big) = \sum_{g \in U \backslash G/W} d(U \cap gWg^{-1}).
\end{equation}

In particular, $\F_p \llb G/W \rrb$ is a finitely related $U$-module if (and only if) there are only finitely many 
$g \in U \backslash G/W$ for which the intersection $U \cap gWg^{-1}$ is nontrivial, and each intersection is a finitely generated pro-$p$ group.

\subsection{Demushkin groups}

Let $G$ be a nonsolvable pro-$p$ Demushkin group.
As mentioned earlier, $G$ is a one-relator group, or more succinctly
\begin{equation} \label{FirstHomEq}
\dim_{\mathbb{F}_p} H_2(G,\mathbb{F}_p) = 1.
\end{equation}

\begin{cor} \label{H2BCor}

For a finite $\F_p \llb G \rrb$-module $L$ we have 
\begin{equation}
\dim_{\F_p} H_2(G,L) \leq \dim_{\F_p} L.
\end{equation}

\end{cor}

\begin{proof}
By picking a $G$-submodule of codimension $1$ we get the exact sequence
\begin{equation}
0 \to L_0 \to L \to \F_p \to 0.
\end{equation}
The associated long exact sequence tells us that
\begin{equation}
\dim_{\F_p} H_2(G,L) \leq \dim_{\F_p} H_2(G,L_0) + \dim_{\F_p} H_2(G,\F_p)
\end{equation}
so our bound follows by induction using equation \eqref{FirstHomEq}.
\end{proof}

The cohomological dimension of $G$ is $2$, so we have the following.

\begin{cor} \label{MonH2Cor}

For any $G$-submodule $L$ of an $\F_p \llb G \rrb$-module $M$ we have
\begin{equation}
\dim_{\F_p} H_2(G, L) \leq \dim_{\F_p} H_2(G, M).
\end{equation}

\end{cor}

\begin{proof}

Consider the short exact sequence of $\F_p \llb G \rrb$-modules
\begin{equation}
0 \to L \to M \to M/L \to 0.
\end{equation}
A part of the associated long exact sequence is
\begin{equation}
H_3(G, M/L) \to H_2(G, L) \to H_2(G, M)
\end{equation}
whose first term vanishes as $G$ is of cohomological dimension $2$.
\end{proof}

Any open subgroup $G_0$ of $G$ is also a nonsolvable Demushkin group, and its number of generators is given by the formula
\begin{equation} \label{RankFor}
d\big(G_0\big) - 2 = \big(d(G) - 2\big)[G : G_0]
\end{equation}
that appears in \cite[Exercise 4.5.6]{Ser}.
By \cite[Exercise 4.5.5]{Ser} or \cite[Theorem 2 (ii)]{Lab0}, any infinite index subgroup of $G$ is free pro-$p$.

%Let $M$ be a finitely presented $\F_p \llb G \rrb$-module with
%\begin{equation} \label{FinH2Eq}
%\dim_{\F_p} H_2(G,M) < \infty.
%\end{equation}
%Recall that the Euler characteristic of $M$ over $G$ is
%\begin{equation}
%\begin{split}
%\chi^G(M) &\defeq \sum_{i \geq 0} (-1)^i\dim_{\F_p} H_i(G,M) \\
%&= \dim_{\F_p} H_0(G,M) -\dim_{\F_p} H_1(G,M) + \dim_{\F_p} H_2(G,M).
%\end{split}
%\end{equation}
%For any short exact sequence of finitely presented $\F_p \llb G \rrb$-modules
%\begin{equation}
%0 \to M_1 \to M_2 \to M_3 \to 0
%\end{equation}
%satisfying inequality \eqref{FinH2Eq}, the Euler characteristic is additive, that is
%\begin{equation}
%\chi^G(M_2) = \chi^G(M_1) + \chi^G(M_3).
%\end{equation}
%This follows immediately from considering the associated long exact sequence.
%In the following we show that the Euler characteristic is index-proportional.
%
%\begin{prop}
%
%For any $K \leq_o G$ we have $\chi^K(M) = [G : K]\chi^G(M).$
%
%\end{prop}
%
%\begin{proof}
%
%Shapiro's lemma tells us that 
%\begin{equation}
%\chi^K \big(M \big) = \chi^G \big( M \ \widehat\otimes_{\F_p} \ \F_p[G/K] \big)
%\end{equation}
%and the module on the right hand side has a filtration of length $[G : K]$ with consecutive quotients isomorphic to $M$. 
%Our formula follows (by induction) from additivity of the Euler characteristic in exact sequences.
%\end{proof}

\section{Finiteness of the set $S$}

The purpose of this section is to show that the set $S$ from \thmref{FirstRes} is finite, 
and to deduce that $\F_p \llb G/U \times G/W \rrb$ is a finitely related $G$-module.

\begin{lem} \label{BettiLem}

Let $G$ be a pro-$p$ Demushkin group with $d(G) \geq 3$, let $A$ be a subgroup of $G$ with 
\begin{equation}
d(A) + 1 < d(G),
\end{equation} and let $T$ be an infinite subset of $G$.
Then there exists a subgroup $B$ of infinite index in $G$ that contains both $A$ and infinitely many elements of $T$.

\end{lem}

\begin{proof}

We inductively construct a strictly descending sequence of subgroups $G_n \leq_o G$, 
and an ascending sequence of subgroups $A_n \leq_c G$ such that:
\begin{enumerate}
\item The inclusion $A_n \subseteq G_n$ and the inequality $d(A_n) + 1 < d(G_n)$ hold.
\item The set $T_n \defeq T \cap G_n$ is infinite.
\item The subgroup $A_n$ contains (at least) $n$ distinct elements from $T$.
\end{enumerate}
Set $A_0 \defeq A, \ G_0 \defeq G$, and suppose that we have completed our construction up to some $n \in \mathbb{N}$ inclusive.
We claim that the index $p$ subgroups of $G_n$ that contain $A_n$ cover $G_n$.
Indeed, let $g \in G_n$ and set 
\begin{equation}
H \defeq \langle A_n \cup \{g\} \rangle.
\end{equation}
By (1) above, we have
\begin{equation}
d(H) \leq d(A_n) + 1 < d(G_n)
\end{equation}
so $H$ is a proper subgroup of $G_n$, and is thus contained in a subgroup of index $p$ in $G_n$. Hence, our claim is verified.

As $G_n$ is an open subgroup of $G$, it has only finitely many subgroups of index $p$ containing $A_n$.
Since these subgroups cover $G_n$, it follows from (2) above that one such subgroup, which we take as our $G_{n+1}$, contains infinitely many elements from $T_n$. 
Using (3) above, we can find a subset
\begin{equation} \label{DefReQ}
R \subseteq A_n \cap T  = A_n \cap T_{n+1}
\end{equation}
such that $|R| = n$. As $T_{n+1}$ is infinite, we can pick a $t \in T_{n+1} \setminus R$, and put
\begin{equation}
A_{n+1} \defeq \langle A_n \cup \{t\} \rangle.
\end{equation}
Recalling equation \eqref{DefReQ} we get that
\begin{equation}
|A_{n+1} \cap T| \geq |A_{n+1} \cap T_{n+1}| \geq |R \cup \{t\}| = n+1.
\end{equation}
Furthermore, from equation \eqref{RankFor} we get that
\begin{equation}
\begin{split}
d \big(G_{n+1}\big) &= \big(d(G_n) - 2\big)p + 2 \geq 2d\big(G_n\big) - 2 \geq d\big(G_n\big) + 1 \\
&> d\big(A_n\big) + 2 \geq d\big(A_{n+1}\big) + 1
\end{split}
\end{equation}
so we have completed our induction.

To conclude, set 
\begin{equation}
B \defeq \overline{\bigcup_{n \in \mathbb{N}} A_n}
\end{equation}
and observe that $B \leq_c G_n$ for each $n \in \mathbb{N}$, so 
\begin{equation}
[G : B] \geq \sup_{n \in \mathbb{N}} \ [G : G_n] = \infty
\end{equation}
as the sequence $G_n$ is strictly descending. At last, by $(1) - (3)$ we have
\begin{equation}
\forall \ n \in \mathbb{N} \quad |B \cap T| \geq |A_n \cap T_n| \geq  n
\end{equation}
so $B$ contains infinitely many elements from $T$, as required.
\end{proof}

In the proof of \cite[Lemma 4.2]{J} it is shown that for any two finitely generated subgroups $U,W$ of a finitely generated free pro-$p$ group $F$, one has
\begin{equation} \label{SfreeFinEq}
\big| \{g \in U \backslash F/ W \ | \ U \cap gWg^{-1} \neq 1\} \big| < \infty.
\end{equation}
In the proof of the following corollary, we shall apply this to a free pro-$p$ group $F$ of countable rank.
This is justified by the embeddability of a free pro-$p$ group of countable rank into a finitely generated free pro-$p$ group.

\begin{cor} \label{PlusOneCor}

Let $G$ be a nonsolvable Demushkin group, and let $U,W$ be subgroups of $G$ such that
\begin{equation}
d(U) + d(W) + 1 < d(G).
\end{equation}
Then the set
\begin{equation}
S \defeq \{g \in U \backslash G/ W \ | \ U \cap gWg^{-1} \neq 1 \}
\end{equation}
is finite.

\end{cor} 

\begin{proof}

Set $A \defeq \langle U \cup W \rangle$, note that
\begin{equation}
d(A) + 1 \leq d(U) + d(W) + 1< d(G),
\end{equation}
and suppose toward a contradiction that $S$ is infinite.
By an abuse of notation, we shall identify $S$ with a section (some set of representatives) of it in $G$.
It follows from \lemref{BettiLem} that there exists a subgroup $B$ of infinite index in $G$ such that
\begin{equation}
U,W \subseteq B, \quad |B \cap S| = \infty.
\end{equation}
It follows immediately that the set
\begin{equation}
\{ g \in U \backslash B/ W \ | \ U \cap gWg^{-1} \neq 1 \}
\end{equation}
is also infinite, contrary to equation \eqref{SfreeFinEq}, as $B$ is a free pro-$p$ group.
\end{proof}

In order to reduce the general case to that of \corref{PlusOneCor}, we need the following claim.

\begin{prop} \label{RedProp}
Let $G$ be a Demushkin group with $d(G) \geq 3$, and let $U,W$ be finitely generated subgroups of infinite index in $G$.
Then there exists an open normal subgroup $N$ of $G$ such that
\begin{equation} \label{NaimEq}
d(U \cap N) + d(W \cap N) + 1 < d(N).
\end{equation}

\end{prop}

\begin{proof}

As $U,W$ are of infinite index in $G$, we can choose an open normal subgroup $N$ of $G$ such that
\begin{equation} \label{NchoiceEq}
[G : UN] \geq 2d(U), \quad [G : WN] \geq 2d(W).
\end{equation}
Since $U$ is free pro-$p$, Schreier's formula gives
\begin{equation}
\begin{split}
d\big(U \cap N\big)  &= \big(d(U)-1\big)[U : U \cap N] + 1 = \big(d(U)-1\big)[UN : N] + 1 \\
&= \frac{\big(d(U)-1\big)[G : N]}{[G : UN]} + 1 \leq \frac{d(U)[G : N]}{[G : UN]}.
\end{split}
\end{equation}
Arguing similarly for $W$ (instead of $U$), and combining the bounds, we infer that the left hand side of equation \eqref{NaimEq} does not exceed
\begin{equation}
\frac{d(U)[G : N]}{[G : UN]} + \frac{d(W)[G : N]}{[G : WN]} + 1.
\end{equation}
Taking into account \eqref{NchoiceEq} and \eqref{RankFor}, we arrive at the desired inequality.
\end{proof}

\begin{cor} \label{FinSCor}

Let $G$ be a nonsolvable Demushkin group, and let $U,W$ be finitely generated subgroups of infinite index in $G$.
Then the set $S$ from \thmref{FirstRes} and \corref{PlusOneCor} is finite.

\end{cor} 

\begin{proof}

Take $N$ to be an open normal subgroup of $G$ as in \propref{RedProp}, and let $R$ be a (finite) set of representatives for the cosets of $N \backslash G$. 

Let $s$ be a representative for some double coset from $S$, and write $s = nr$ for some $n \in N$ and $r \in R$.
It follows at once from inequality \eqref{NaimEq} that
\begin{equation}
d(U \cap N) + d(rWr^{-1} \cap N) + 1 < d(N)
\end{equation}
so we infer from \corref{PlusOneCor}, and from torsion-freeness of $G$, that the set
\begin{equation}
S_r \defeq \big\{m \in (U \cap N)\backslash N/ (rWr^{-1} \cap N)\ | \ U \cap mrWr^{-1}m^{-1} \neq 1 \big\}
\end{equation}
is finite. Let $L_r$ be a (finite) set of representatives for the double cosets in $S_r$.
By our choice of $s$, we have
\begin{equation}
U \cap nrWr^{-1}n^{-1} \neq 1,
\end{equation}
so $n$ represents a double coset from $S_r$. Hence, $u \cdot n \cdot rwr^{-1} \in L_r$ for some $u \in U, \ w \in W$.
This means that $usw \in L_r \cdot r$, so every double coset in $S$ can be represented by an element from $L_r \cdot r$ for some $r \in R$.
\end{proof}

We now combine the finiteness of $S$ with Howson's theorem for Demushkin groups (see \cite[Theorem 1.8]{ShZ}) into a single homological statement.

\begin{cor} \label{HomRefCor}

Let $G$ be a nonsolvable pro-$p$ Demushkin group, and let $U,W$ be finitely generated infinite index subgroups. Then
\begin{equation} \label{FirstTorAimEq}
\dim_{\F_p} H_1 \big(G, \F_p \llb G/U \rrb \ \widehat\otimes_{\F_p} \ \F_p \llb G/W \rrb \big) = \sum_{g \in U \backslash G/W} d(U \cap gWg^{-1}) 
\end{equation}
is finite.

\end{cor}

\begin{proof}

By Shapiro's lemma, our homology group is isomorphic to
\begin{equation}
H_1 \big( U, \F_p \llb G/W \rrb \big)
\end{equation}
and its dimension is calculated in equation \eqref{Eq3}.
By \corref{FinSCor}, there are only finitely many nonzero summands on the right hand side of equation \eqref{FirstTorAimEq}, 
and \cite[Theorem 1.8]{ShZ} tells us that each summand is finite.
\end{proof}

\section{The Relation gradient}

\corref{HomRefCor} gives a homological interpretation of a sum very similar to the one appearing in the Hanna Neumann conjecture. In order to write the required sum (from the left hand side of equation \eqref{HNG}) in a homological form, we introduce a homological gradient (analogous to a Betti number).

\begin{defi} 

Let $G$ be a pro-$p$ group, and let $M$ be a finitely related $\F_p \llb G \rrb$-module.
We set
\begin{equation}
\beta_1^G(M) \defeq \inf_{H \leq_o G} \frac{\dim_{\F_p} H_1(H,M)}{[G : H]}.
\end{equation}

\end{defi}
We call this nonnegative real number the relation gradient of $M$ over $G$.
In fact, we can restrict ourselves (in the infimum above) to normal subgroups, or (more generally) to any cofinal family of open subgroups.
This follows from the following folklore lemma.

\begin{lem} \label{FolkLem}
Let $G$ be a pro-$p$ group, let $M$ be an $\F_p \llb G \rrb$-module, and let $K \leq H$ be open subgroups of $G$. Then, for any $n \geq 0$, we have
\begin{equation} \label{FolkLemAim}
\frac{\dim_{\F_p} H_n(K,M)}{[G : K]} \leq \frac{\dim_{\F_p} H_n(H,M)}{[G:H]}.
\end{equation}
\end{lem}

\begin{proof}
By Shapiro's lemma,
\begin{equation} \label{ItIndEq}
\dim_{\F_p} H_n \big( K, M \big) = \dim_{\F_p} H_n \big( H, M \ \widehat\otimes_{\F_p} \ \F_p[H/K] \big)
\end{equation}
and the $H$-module 
\begin{equation} 
M \ \widehat\otimes_{\F_p} \ \F_p[H/K]
\end{equation}
admits a filtration (by $H$-submodules) 
of length $[H : K]$ with consecutive quotients isomorphic to $M$.
Hence, the bounds coming from the long exact sequences (associated to our filtration) yield
\begin{equation} 
\dim_{\F_p} H_n \big( H, M \ \widehat\otimes_{\F_p} \ \F_p[H/K] \big) \leq 
[H : K]\dim_{\F_p} H_n \big( H, M \big).
\end{equation}
Dividing by $[G : K]$ and recalling equation \eqref{ItIndEq}, we finish the proof.
\end{proof}

The family of those open subgroups of a profinite group $G$ that are contained in a given open subgroup $H$ of $G$ is clearly cofinal. As a result, we obtain the index-proportionality of the relation gradient.

\begin{cor} \label{IPCor}

For a pro-$p$ group $G$, an open subgroup $H$ of $G$, and a finitely related $\F_p \llb G \rrb$-module $M$ we have 
$\beta_1^H(M) = [G : H]\beta_1^G(M)$.

\end{cor}

\begin{proof}

From the definition, we get
\begin{equation}
\beta_1^H(M) = \inf_{K \leq_o H} \frac{\dim_{\F_p} H_1(K,M)}{[H : K]} = 
[G : H] \inf_{K \leq_o H} \frac{\dim_{\F_p} H_1(K,M)}{[G : K]}
\end{equation}
and by cofinality, the latter expression equals $[G : H]\beta_1^G(M)$.
\end{proof}

The following proposition is a `Shapiro lemma' for the relation gradient.

\begin{prop} \label{b1Shap}

Let $G$ be a pro-$p$ group, let $U \leq_c G$, and let $M$ be a finitely related $\F_p \llb U \rrb$-module. Then 
\begin{equation} \label{Ab1Shap}
\beta_1^G(\mathrm{Ind}_U^G M) = \beta_1^U(M).
\end{equation}

\end{prop}

\begin{proof}
From transitivity of induction, we get
\begin{equation} \label{4.8Eq}
\beta_1^G(\mathrm{Ind}_U^G M) = 
\inf_{H \lhd_o G} \frac{\dim_{\F_p} H_1(H, \mathrm{Ind}_{HU}^G\mathrm{Ind}_U^{HU} M  )}{[G : H]}.
\end{equation}
Normality of $H$ in $G$ implies that for any $\F_p \llb HU \rrb$-module $N$ we have
\begin{equation}
\mathrm{Ind}_{HU}^G N \cong_H N^{\oplus [G : HU]}
\end{equation}
so taking $N = \mathrm{Ind}_U^{HU} M$ and using the fact that homology commutes with direct sums, 
we see that the right hand side of equation \eqref{4.8Eq} simplifies to
\begin{equation}
\inf_{H \lhd_o G} \frac{\dim_{\F_p} H_1 \big( H, \mathrm{Ind}_U^{HU} M \big)}{[HU : H]}.
\end{equation}
Using (for instance) Mackey's formula, the expression above becomes
\begin{equation}
\inf_{H \lhd_o G} \frac{\dim_{\F_p} H_1 \big( H, \mathrm{Ind}_{H \cap U}^{H} M \big)}{[U : H \cap U]}.
\end{equation}
By Shapiro's lemma, our infimum is just
\begin{equation}
\inf_{H \lhd_o G} \frac{\dim_{\F_p} H_1 \big( H \cap U, M \big)}{[U : H \cap U]}
\end{equation}
so from the cofinality of $\{H \cap U : H \lhd_o G\}$ among the open subgroups of $U$,
we conclude that the infimum above evaluates to $\beta_1^U(M)$.
\end{proof}

The following establishes the subadditivity of the relation gradient in short exact sequences.

\begin{prop} \label{SaProp}

Let $G$ be a pro-$p$ group, and let
\begin{equation} \label{SESAprop}
0 \to M_1 \to M_2 \to M_3 \to 0
\end{equation}
be a short exact sequence of $\F_p \llb G \rrb$-modules, with $M_1$ and $M_3$ finitely related.
Then $M_2$ is finitely related as well, and
\begin{equation} \label{SaEq}
\beta_1^G(M_2) \leq \beta_1^G(M_1) + \beta_1^G(M_3).
\end{equation}
Moreover, if the short exact sequence splits, then
\begin{equation} \label{additEq}
\beta_1^G(M_2) = \beta_1^G(M_1) + \beta_1^G(M_3).
\end{equation}

\end{prop}

\begin{proof}

From the long exact sequence associated to \eqref{SESAprop}, we get that
\begin{equation}
\dim_{\F_p} H_1(G, M_2) \leq \dim_{\F_p} H_1(G, M_1) + \dim_{\F_p} H_1(G, M_3) < \infty
\end{equation}
as $M_1, M_3$ are finitely related. Hence, $M_2$ is a finitely related $G$-module.
 
Let $\epsilon > 0$, and for each $i \in \{1, 2, 3\}$ pick some $K_i \leq_o G$ such that
\begin{equation}
\frac{\dim_{\F_p} H_1(K_i, M_i)}{[G : K_i]} \leq \beta_1^G(M_i) + \frac{\epsilon}{2}.
\end{equation}
Setting $K \defeq K_1 \cap K_2 \cap K_3$ we see that for every $i \in \{1, 2, 3\}$ we still have
\begin{equation} \label{HEpsEq}
\frac{\dim_{\F_p} H_1(K, M_i)}{[G : K]} \leq \beta_1^G(M_i) + \frac{\epsilon}{2}
\end{equation}
in light of \lemref{FolkLem}.
The aforementioned long exact sequence now gives
\begin{equation}
\beta_1^G(M_2) \leq \frac{\dim_{\F_p} H_1(K, M_2)}{[G : K]} \leq
\sum_{i=1,3} \frac{\dim_{\F_p} H_1(K, M_i)}{[G : K]}
\end{equation}
so applying inequality \eqref{HEpsEq} to the right hand side we obtain \eqref{SaEq}.

Suppose that our exact sequence splits. By inequality \eqref{HEpsEq} we have
\begin{equation}
\begin{split}
\beta_1^G(M_2) &\geq \frac{\dim_{\F_p} H_1(K, M_2)}{[G : K]} - \epsilon \\
&= \frac{\dim_{\F_p} H_1(K, M_1)}{[G : K]} + \frac{\dim_{\F_p} H_1(K, M_3)}{[G : K]} - \epsilon \\
&\geq \beta_1^G(M_1) + \beta_1^G(M_3) - \epsilon
\end{split}
\end{equation}
so combining this with inequality \eqref{SaEq} we arrive at equation \eqref{additEq}.
\end{proof}

For Demushkin groups, the relation gradient enjoys monotonicity.

\begin{prop} \label{MonProp}

Let $G$ be a nonsolvable pro-$p$ Demushkin group, let $M$ be a finitely related $\F_p \llb G \rrb$-module,
and let $N$ be a $G$-submodule of $M$.
Suppose that $M/N$ is finite, or that $H_2(G, M/N) = 0$. 
Then $\beta_1^G(N) \leq \beta_1^G(M)$.

\end{prop}

\begin{proof}

Let $\{G_n\}_{n=1}^\infty$ be a descending sequence of open subgroups of $G$ intersecting trivially.
Such a sequence is cofinal, so by \lemref{FolkLem} we have
\begin{equation} \label{ChainEq}
\beta_1^G(N) = \inf_{n \geq 1} \frac{\dim_{\F_p} H_1(G_n,N)}{[G : G_n]} =
\lim_{n \to \infty} \frac{\dim_{\F_p} H_1(G_n,N)}{[G : G_n]}.
\end{equation}
The long exact sequence associated to the short exact sequence
\begin{equation}
0 \to N \to M \to M/N \to 0
\end{equation}
tells us that the rightmost part of equation \eqref{ChainEq} does not exceed
\begin{equation} \label{2SumdsEq}
\lim_{n \to \infty} \frac{\dim_{\F_p} H_1(G_n,M)}{[G : G_n]} + 
\lim_{n \to \infty} \frac{\dim_{\F_p} H_2(G_n,M/N)}{[G : G_n]}.
\end{equation}
As the first summand above equals $\beta_1^G(M)$, we need to show that the second summand vanishes.
If $M/N$ is finite, this follows from \corref{H2BCor},
while if $H_2(G, M/N) = 0$ we can use \lemref{FolkLem}.
%\begin{equation}
%0 \to R_2 \to R_1 \to M/N \to 0
%\end{equation}
%of $\F_p \llb G \rrb$-modules, with $R_1$ free.
%Considering the associated long exact sequence, 
%we see that our assumption implies that $R_2$ is free as well.
%Hence, for every $n \geq 1$ we get from the long exact sequence that
%\begin{equation}
%H_2(G_n, M/N) = 0
%\end{equation}
%so the second summand of \eqref{2SumdsEq} vanishes, as required.
\end{proof}

The reason for introducing the relation gradient is seen from the next corollary.

\begin{cor} \label{Reason}

In the notation of \thmref{FirstRes}, with $[G : U], [G : V]$ infinite, we have
\begin{equation} \label{ReasonAim}
\begin{split}
&\beta_1^G \big( \F_p \llb G/U \rrb \big ) = d(U) - 1, \quad  \beta_1^G \big( \F_p \llb G/W \rrb \big ) = d(W) - 1, \\
&\beta_1^G \big(  \F_p \llb G/U \rrb \ \widehat\otimes_{\F_p} \ \F_p \llb G/W \rrb \big ) = 
\sum_{g \in S} \big( d(U \cap gWg^{-1}) - 1 \big).
\end{split}
\end{equation}
In particular, in order to obtain \thmref{FirstRes}, it suffices to show that
\begin{equation}
\beta_1^G \big(  \F_p \llb G/U \rrb \ \widehat\otimes_{\F_p} \ \F_p \llb G/W \rrb \big ) \leq 
\beta_1^G \big( \F_p \llb G/U \rrb \big )\beta_1^G \big( \F_p \llb G/W \rrb \big ).
\end{equation}
\end{cor}

\begin{proof}

Invoking \propref{b1Shap} with $M = \F_p$, we get
\begin{equation}
\beta_1^G \big( \F_p \llb G/U \rrb \big ) = \beta_1^U(\F_p) = d(U) - 1 
\end{equation}
where the second equality is a consequence of Schreier's formula for the nontrivial free pro-$p$ group $U$.

In the proof of \corref{HomRefCor} we have seen that $\F_p \llb G/W \rrb$ is a finitely related $U$-module,
so by \propref{b1Shap} we have
\begin{equation}
\beta_1^G \big(  \F_p \llb G/U \rrb \ \widehat\otimes_{\F_p} \ \F_p \llb G/W \rrb \big ) =
\beta_1^U \big(  \F_p \llb G/W \rrb \big ).
\end{equation}
In light of equation \eqref{UDecIs} the relation gradient above equals
\begin{equation}
\beta_1^U \Bigg( \bigoplus_{g \in U \backslash G/W} \F_p \llb U/ U \cap gWg^{-1} \rrb \Bigg)
\end{equation}
which reduces, by the split case of \propref{SaProp}, to
\begin{equation} \label{SumOverS}
\sum_{g \in S} \beta_1^U \big(  \F_p \llb U/ U \cap gWg^{-1} \rrb \big).
\end{equation}
The subgroup $U \cap gWg^{-1}$
is finitely generated, so $\F_p$ is a finitely related module for it.
Applying \propref{b1Shap} to each term in the sum gives
\begin{equation}
\sum_{g \in S} \beta_1^{U \cap gWg^{-1}} \big(  \F_p \big)
\end{equation}
so using Schreier's formula as before, the sum above becomes
\begin{equation}
\sum_{g \in S} \big( d(U \cap gWg^{-1}) - 1 \big)
\end{equation}
as required.
\end{proof}

\section{Integrality}

The goal of this section is to show that the relation gradient of any finitely presented $\F_p \llb G \rrb$-module is an integer when $G$ is a nonsolvable pro-$p$ Demushkin group.

\subsection{Pro-$\mathcal{C}$ groups}

Let $\mathcal{C}$ be a class of pro-$p$ groups closed under taking subgroups and under forming extensions. 
This means that
\begin{enumerate}

\item if $G \in \mathcal{C}$ and $H \leq_c G$ then $H \in \mathcal{C}$;

\item if $G$ is a pro-$p$ group and $N \lhd_c G$ is such that $N, G/N \in \mathcal{C}$ then $G \in \mathcal{C}$.

\end{enumerate}
The second conditions implies that $\mathcal{C}$ is closed under taking direct products (of finitely many groups).
Combining this with the first condition, 
we get that $\mathcal{C}$ is also closed under taking fibered (or subdirect) products.
In other words, if $G$ is a pro-$p$ group with $M,N \lhd_c G$ such that $G/M, G/N \in \mathcal{C}$ then $G/ M \cap N \in \mathcal{C}$ as well.
One example is the class of torsion-free poly-procyclic pro-$p$ groups, 
and a larger one is the class of torsion-free $p$-adic analytic groups.

We say that a finitely generated pro-$p$ group $G$ is residually $\mathcal{C}$ (or pro-$\mathcal{C}$) if for every $g \in G$ there exists a homomorphism $\varphi \colon G \to Q$ with $Q \in \mathcal{C}$ such that $g \notin \mathrm{Ker}(\varphi)$. Equivalently, there exists a trivially intersecting chain of subgroups $\Omega_n \lhd_c G$ such that $G/ \Omega_n \in \mathcal{C}$. This is the same as saying that $G$ is an inverse limit of groups from $\mathcal{C}$.
The following proposition gives a criterion for a pro-$p$ group to be residually $\mathcal{C}$.

\begin{prop} \label{CCrit}

Let $G$ be a finitely generated pro-$p$ group, and let 
\begin{equation}
G = G_0 \geq_o G_1 \geq_o G_2  \geq_o \dots \geq_o G_n  \geq_o \cdots
\end{equation}
be a chain of open normal subgroups of $G$ with
\begin{equation} \label{IntTrivEq}
\bigcap_{n=0}^\infty G_n = 1.
\end{equation}
Suppose that for every $n \geq 1$ there exists a subgroup $Q_n \lhd_c G_{n-1}$ such that $Q_n \leq_c G_n$ and
$G_{n-1}/Q_n \in \mathcal{C}$.
Then $G$ is residually $\mathcal{C}$. 

\end{prop} 

\begin{proof}

For $n \geq 1$ let 
\begin{equation}
\widetilde{Q}_n \defeq \bigcap_{g \in G/G_{n-1}} gQ_ng^{-1} \lhd_c G
\end{equation}
be the normal core of $Q_n$ in $G$.
The group $G_{n-1}/ \widetilde{Q}_n$ is a fibered product of the finitely many groups
\begin{equation}
G_{n-1}/ gQ_ng^{-1}, \quad g \in G/G_{n-1}.
\end{equation}
Each of these groups is isomorphic to $G_{n-1}/Q_n \in \mathcal{C}$ so we conclude that $G_{n-1}/ \widetilde{Q}_n \in \mathcal{C}$ as well.
For every $n \geq 1$ set
\begin{equation}
\Omega_n \defeq \bigcap_{i \leq n} \widetilde{Q}_i \lhd_c G
\end{equation}
and note that by equation \eqref{IntTrivEq}, this is a chain of subgroups that satisfies
\begin{equation}
\bigcap_{n = 1}^\infty \Omega_n \subseteq \bigcap_{n = 1}^\infty \widetilde{Q}_n
\subseteq \bigcap_{n = 1}^\infty Q_n 
\subseteq \bigcap_{n = 1}^\infty G_{n-1} = 1.
\end{equation}

We shall argue, by induction on $n$, that $G/\Omega_n \in \mathcal{C}$.
For $n=1$ we have
\begin{equation}
G/\Omega_1 = G/\widetilde{Q}_1 = G_0/\widetilde{Q}_1 \in \mathcal{C}.
\end{equation}
Once $n \geq 2$ we have
\begin{equation}
\Omega_n = \Omega_{n-1} \cap \widetilde{Q}_{n}
\end{equation}
and by induction $G/\Omega_{n-1} \in \mathcal{C}$, so since $\mathcal{C}$ is closed under forming extensions,
it suffices to show that $\Omega_{n-1}/\Omega_{n} \in \mathcal{C}$.
Indeed, since $Q_{n-1} \leq_c G_{n-1}$ we have
\begin{equation}
\Omega_{n-1}/\Omega_{n} = \Omega_{n-1}/\Omega_{n-1} \cap \widetilde{Q}_{n} \cong
\Omega_{n-1}\widetilde{Q}_{n} / \widetilde{Q}_{n} \leq_c G_{n-1}/ \widetilde{Q}_{n} \in \mathcal{C}
\end{equation} 
so we conclude by recalling that $\mathcal{C}$ is closed under taking subgroups.
\end{proof}

\begin{cor}

Let $\mathcal{C}$ be the class of torsion-free poly-procyclic pro-$p$ groups, and let $G$ be a torsion-free pro-$p$ Demushkin group. Then $G$ is residually $\mathcal{C}$.

\end{cor}

\begin{proof}

Let $G_n$ be the Frattini series of $G$, given by
\begin{equation}
G_0 \defeq G, \quad G_n \defeq \Phi(G_{n-1}) = G_{n-1}^p[G_{n-1},G_{n-1}], \quad n \geq 1.
\end{equation}
This is a chain of open normal subgroups of $G$ that intersects trivially.

Fix $n \geq 1$, put $\Gamma \defeq G_{n-1}$, and recall that $\Gamma$ is a torsion-free Demushkin group.
It follows from the classification given in \cite[Theorem 3]{Lab} that
\begin{equation}
\Gamma = 
\langle 
x_1, x_2, \dots, x_m \ | \ x_1^{\alpha_1} x_2^{\alpha_2} [x_{\ell-1}, x_{\ell}]  x_3^{\alpha_3}  P = 1  
\rangle
\end{equation}
as a pro-$p$ group. Here,
\begin{equation}
m \geq 2,\ell, \quad \alpha_1, \alpha_2, \alpha_3 \in \mathbb{Z}_p, 
\quad \ell \in \{2,3\}, \quad \alpha_\ell = 0, \quad \alpha_{\ell-1} \in p\mathbb{Z}_p,
\end{equation}
($\alpha_3 = 0$ if $m = 2$), and $P$ is a product of elements from the set
\begin{equation}
S \defeq \big \{[x_i,x_j] \ | \ 1 \leq i < j \leq m \big \} \setminus \big \{[x_{\ell-1}, x_{\ell}] \big \}.
\end{equation}

Let $Q$ be the closed normal subgroup of $\Gamma$ generated by $S$.
Evidently, $Q$ is contained in $\Phi(\Gamma) = G_n$.
Rewriting the relations slightly, we find that the group $\Gamma/Q$ has a presentation with generators $x_1, \dots, x_m$
and relations
\begin{itemize}

\item $\forall \ i,j \neq \ell \quad [x_i, x_j] = 1;$

\item $\forall \ k \neq \ell-1 \quad x_\ell^{-1} \cdot x_k \cdot x_\ell = x_k;$

\item $x_\ell^{-1} \cdot x_{\ell-1} \cdot x_\ell = x_{\ell-1}x_2^{-\alpha_2}x_1^{-\alpha_1}x_3^{-\alpha_3}.$

\end{itemize}

Since $\alpha_\ell = 0$ and $\alpha_{\ell-1} \in p\mathbb{Z}_p$, the relations above imply that conjugation by $x_\ell$ induces a unipotent endomorphism of the mod $p$ reduction of
\begin{equation}
L \defeq \langle x_i \ | \ 1 \leq i \leq m, \ i \neq \ell \rangle.
\end{equation}
It follows that conjugation by $x_\ell$ is a pro-$p$ automorphism of $L$, that is
\begin{equation}
\Gamma/Q \cong  L \rtimes \langle x_\ell \rangle \cong
\mathbb{Z}_p^{m-1} \rtimes \mathbb{Z}_p \in \mathcal{C}.
\end{equation}
Hence, our corollary follows from \propref{CCrit}.
\end{proof}

In particular, nonsolvable Demushkin groups are residually torsion-free $p$-adic analytic.
For other pro-$p$ groups that are residually torsion-free poly-procyclic see \cite[Theorem 4.2]{KZ2}.

\subsection{The Atiyah conjecture}

It is convenient for us to state the Atiyah conjecture using a variant of the relation gradient.
\begin{defi}

For a pro-$p$ group $G$ and a finitely generated $\F_p \llb G \rrb$-module $M$, we define the rank gradient of $M$ over $G$ to be 
\begin{equation}
\beta_0^G(M) \defeq \inf_{H \leq_o G} \frac{\dim_{\F_p} H_0(H,M)}{[G : H]}.
\end{equation}

\end{defi}
The rank gradient behaves in a manner similar to the relation gradient, and in particular, 
\lemref{FolkLem} holds for it.

The Atiyah conjecture states that the rank gradient is an integer once $G$ belongs to the class of torsion-free pro-$p$ groups. For other forms of the conjecture see \cite{J2}.
We are interested in this conjecture in light of the following.

\begin{prop} \label{ArgProp}

Let $G$ be torsion-free pro-$p$ group for which the Atiyah conjecture holds.
Then the relation gradient of any finitely presented $\F_p \llb G \rrb$-module $M$ is an integer.

\end{prop}

\begin{proof}

As $M$ is finitely presented, there exists a short exact sequence
\begin{equation}
0 \to K \to \F_p \llb G \rrb^d \to M \to 0
\end{equation}
of $\F_p \llb G \rrb$-modules with $d \in \mathbb{N}$, and $K$ finitely generated.
For any $H \leq_o G$,
considering the associated long exact sequence we see that
\begin{equation}
\frac{\dim_{\F_p} H_1(H, M)}{[G : H]} = 
\frac{\dim_{\F_p} H_0(H, K)}{[G : H]} - d + \frac{\dim_{\F_p} H_0(H, M)}{[G : H]}.
\end{equation}
At last, take the infimum over all $H \leq_o G$ and use the Atiyah conjecture.
\end{proof}

For the class $\mathcal{C}$ of torsion-free $p$-adic analytic groups, a proof of the Atiyah conjecture (based on ideas by Lazard, Harris, and Farkas-Linnell) is given in \cite[Theorem 2.1]{BLLS}.
From that, we deduce the following.
\begin{cor}

A finitely generated residually $\mathcal{C}$ pro-$p$ group $G$ satisfies the Atiyah conjecture.

\end{cor}

\begin{proof}

Let $M$ be a finitely generated $\F_p \llb G \rrb$-module, let $\{G_n\}_{n = 0}^\infty$ be a trivially intersecting sequence of open subgroups of $G$, and let $\{\Omega_k\}_{k = 0}^\infty$ be a trivially intersecting chain of normal subgroups of $G$ with $G/\Omega_k \in \mathcal{C}$.

We inductively construct a chain $\{H_n\}_{n = 0}^\infty$ of open subgroups of $G$.
First set $H_0 = G_0$ and suppose that $H_i$ has already been defined for $i < n$.
Pick an integer $k = k(n)$ such that $\Omega_{k} \subseteq H_{n-1} \cap G_n$, and choose $H_n$ to be an open subgroup of 
$H_{n-1} \cap G_n$ that contains $\Omega_k$ and satisfies
\begin{equation}
\Bigg|  \frac{\dim_{\F_p} H_0(H_n/\Omega_k, M_{\Omega_k})}{[G/\Omega_k : H_n/\Omega_k]} -  \beta_0^{G/\Omega_{k}}(M_{\Omega_{k}}) \Bigg| \leq \frac{1}{n}.
\end{equation}
As the Atiyah conjecture for $G/\Omega_k$ is true, the inequality above reduces to
\begin{equation} \label{SimpEq}
\Bigg|  \frac{\dim_{\F_p} H_0(H_n, M)}{[G : H_n]} -  z_k \Bigg| \leq \frac{1}{n}
\end{equation}
for some $z_k \in \mathbb{Z}$.
It follows from our construction of the chain that
\begin{equation}
\bigcap_{n=0}^\infty H_n \subseteq \bigcap_{n=0}^\infty G_n = 1
\end{equation}
so the chain is cofinal, and thus (by \lemref{FolkLem}) we have
\begin{equation}
\beta_0^G(M) = \inf_{n \geq 0} \frac{\dim_{\F_p} H_0(H_n, M)}{[G : H_n]} = 
\lim_{n \to \infty} \frac{\dim_{\F_p} H_0(H_n, M)}{[G : H_n]}
\end{equation} 
which is arbitrarily close to an integer, by inequality \eqref{SimpEq}.
\end{proof}
In particular, nonsolvable Demushkin groups satisfy the Atiyah conjecture.
As a result, the values of $\beta_1$ on finitely presented modules are integral.

\section{One-relator modules}

Let $G$ be a pro-$p$ group.
We say that an $\F_p \llb G \rrb$-module $M$ is a one-relator module if it is a finitely generated $G$-module that satisfies
\begin{equation}
\dim_{\F_p} H_1(G, M) = 1.
\end{equation}
This is equivalent to the existence of a short exact sequence
\begin{equation} \label{OneRelEq}
0 \to C \to \F_p \llb G \rrb^d \to M \to 0
\end{equation}
where $d = \dim_{\F_p} H_0(G,M)$ and $C$ is a nontrivial cyclic $\F_p \llb G \rrb$-module.

We shall need a Schreier formula characterization of freeness.

\begin{lem} \label{SchreierLem}

Let $G$ be a pro-$p$ group, and let $M$ be a finitely generated $\F_p \llb G \rrb$-module.
Then $M$ is a free $\F_p \llb G \rrb$-module if and only if
\begin{equation}
\beta_0^G(M) = \dim_{\F_p} H_0(G,M).
\end{equation}

\end{lem}

\begin{proof}

By definition of the rank gradient, we need to show that the equality
\begin{equation} \label{ModuleSchreierEq}
\dim_{\F_p} H_0(K,M) = [G : K] \dim_{\F_p} H_0(G,M)  
\end{equation}
holds for every open subgroup $K$ of $G$, if and only if $M$ is free.
Indeed, if $M$ is a free $\F_p \llb G \rrb$-module of rank $d \defeq \dim_{\F_p} H_0(G,M)$,
then as an $\F_p \llb K \rrb$-module,
$M$ is a direct sum of $[G : K]$ copies of a free $\F_p \llb K \rrb$-module of rank $d$.
Hence, equation \eqref{ModuleSchreierEq} holds in this case.

For the other direction, write an exact sequence of $\F_p \llb G \rrb$-modules
\begin{equation} \label{ExSecForM}
0 \to N \to \F_p \llb G \rrb^d \to M \to 0
\end{equation}
where (as previously) $d = \dim_{\F_p} H_0(G,M)$.
Equation \eqref{ModuleSchreierEq} implies that
\begin{equation}
H_0 \big( K, \F_p \llb G \rrb^d \big) \to H_0 \big(K, M \big)
\end{equation}
is an injection for any $K \lhd_o G$, or equivalently, that the map
\begin{equation}
H_0 \big( K, N \big) \to H_0 \big( K, \F_p \llb G \rrb^d \big) \cong \F_p [G/K]^d
\end{equation}
is zero.
We conclude that $N$ is contained in the kernel of the map
\begin{equation}
\F_p \llb G \rrb^d \to \F_p [G/K]^d
\end{equation}
for every $K \lhd_o G$.
The intersection of these kernels is trivial, so $N = 0$, and thus from equation \eqref{ExSecForM} we infer that $M \cong \F_p \llb G \rrb^d$ as required.
\end{proof}

The vanishing results in the next section are based on the following.

\begin{cor} \label{NoBeta1For1Relator}

Let $G$ be a torsion-free pro-$p$ group that satisfies Atiyah's conjecture,
and let $M$ be a one-relator $\F_p \llb G \rrb$-module.
Then $\beta_1^G(M) = 0$.

\end{cor}

\begin{proof}

As $M$ is a one-relator module, we have a short exact sequence
\begin{equation} 
0 \to C \to \F_p \llb G \rrb^d \to M \to 0
\end{equation}
of $\F_p \llb G \rrb$-modules with
\begin{equation} \label{StuEq}
\dim_{\F_p} H_0(G, C) = 1, \quad \dim_{\F_p} H_0(G, M) = d.
\end{equation}

Arguing as in the proof of \propref{ArgProp}, we find that our short exact sequence gives the equality
\begin{equation} \label{B1AsB0Eq}
\beta_1^G(M) = \beta_0^G(C) - d + \beta_0^G(M).
\end{equation}
Since $M$ is not a free $\F_p \llb G \rrb$-module, \lemref{SchreierLem} tells us that
\begin{equation}
\beta_0^G(M) < \dim_{\F_p} H_0(G, M).
\end{equation} 
As $G$ satisfies Atiyah's conjecture, $\beta_0^G(M)$ is an integer so the above becomes
\begin{equation}
\beta_0^G(M) \leq \dim_{\F_p} H_0(G, M) - 1.
\end{equation}
Combining this with equation \eqref{B1AsB0Eq}, and using equation \eqref{StuEq}, we get
\begin{equation}
\beta_1^G(M) \leq \beta_0^G(C) - d + \dim_{\F_p} H_0(G, M) - 1 = \beta_0^G(C) - 1 \leq 0
\end{equation}
where the last inequality holds since $\beta_0^G(C) \leq \dim_{\F_p} H_0(G, C) = 1$.
\end{proof}

\begin{cor}

Let $G$ be a torsion-free pro-$p$ group that satisfies Atiyah's conjecture,
and let $a,b \in \F_p \llb G \rrb$ be nonzero. Then $ab \neq 0$.

\end{cor}

\begin{proof}

Our statement is obvious if $b$ is a unit, so we assume that this is not the case.
Let $C$ be the cyclic submodule of $\F_p \llb G \rrb$ generated by $b$, 
and let $M$ be the one-relator $G$-module $\F_p \llb G \rrb/C$.
Equation \eqref{B1AsB0Eq} (with $d = 1$) reads
\begin{equation} \label{CGradEq}
\beta_0^G(C) = \beta_1^G(M) + 1 - \beta_0^G(M) \geq 1 - \beta_0^G(M).
\end{equation}
The $\F_p \llb G \rrb$-module $M$ is not free, so by \lemref{SchreierLem} we know that
\begin{equation}
\beta_0^G(M) < \dim_{\F_p} H_0(G,M) = 1.
\end{equation}
Moreover, $\beta_0^G(M)$ is an integer as $G$ satisfies Atiyah's conjecture.
We conclude that  $\beta_0^G(M) = 0$ so equation \eqref{CGradEq} says that 
\begin{equation}
\beta_0^G(C) \geq 1 = \dim_{\F_p} H_0(G, C).
\end{equation}
Invoking \lemref{SchreierLem} once again, we get that $C$ is a free $\F_p \llb G \rrb$-module.
In other words, the annihilator of $b$ is trivial, so $ab \neq 0$. 
\end{proof}

\section{Vanishing}

Our goal here is to establish,
for a finitely generated subgroup $U$ of a Demushkin group $G$, 
the vanishing of the relation gradient for an open submodule of $\F_p \llb G/U \rrb$ with codimension as small as possible.

We begin with a quite general vanishing lemma.

\begin{lem} \label{VirtuallyFreeLem}

Let $G$ be a pro-$p$ group, let $U$ be a free pro-$p$ subgroup of $G$, and let $M$ be an $\F_p \llb G \rrb$-module that is finitely related over $U$.
Then there exists an open $G$-submodule $M_0$ of $M$ such that $\beta_1^U(M_0) = 0$.

\end{lem}

\begin{proof}

The group $H_1(U,M)$ is finite, 
so by \cite[Proposition 6.5.7]{RZ}, there exists an open $G$-submodule $M_0$ of $M$ such that the map
\begin{equation} \label{Inv0LimEq}
H_1 (U, M) \to H_1 (U, M/M_0)
\end{equation}
is injective. 
Consider the short exact sequence
\begin{equation}
0 \to M_0 \to M \to M/M_0 \to 0
\end{equation}
of $U$-modules.
For the associated long exact sequence, the aforementioned injectivity means that the connecting homomorphism
\begin{equation}
H_2 \big( U, M/M_0) \to H_1 (U, M_0)
\end{equation}
is surjective. 
Since $U$ is free, all second homology groups vanish, so we get
\begin{equation}
\beta_1^U(M_0) \leq \dim_{\F_p} H_1(U, M_0) \leq \dim_{\F_p} H_2(U, M/M_0) = 0
\end{equation}
as required.
\end{proof}

In the following proposition we obtain the vanishing of the relation gradient for a submodule of finite (but ineffective) codimension in $\F_p \llb G/U \rrb$.

\begin{prop} \label{IneffProp}

Let $G$ be a nonsolvable pro-$p$ Demushkin group, and let $U$ be a finitely generated subgroup of $G$.
Then the $G$-module $\F_p \llb G/U \rrb$ has an open $G$-submodule $M$ with $\beta_1^G(M) = 0$.

\end{prop}

\begin{proof}

Let $H$ be an open subgroup of $G$ containing $U$ such that the  map
\begin{equation} \label{InvLimEq}
H_1 \big( G, \F_p \llb G/U \rrb \big) \to H_1 \big( G, \F_p [G/H] \big)
\end{equation}
is injective, and let $M$ be the (unique) $G$-submodule of $\F_p \llb G/U \rrb$ that fits into the short exact sequence
\begin{equation}
0 \to M \to \F_p \llb G/U \rrb \to \F_p [G/H] \to 0.
\end{equation}
Considering the associated long exact sequence, injectivity in equation \eqref{InvLimEq} implies that the connecting homomorphism
\begin{equation}
H_2 \big( G, \F_p [G/H] \big) \to H_1 \big( G, M \big)
\end{equation}
is surjective. Hence, using Shapiro's lemma, we find that
\begin{equation} \label{OneRelatorEq}
\dim_{\F_p} H_1 \big( G, M \big) \leq \dim_{\F_p} H_2 \big( G, \F_p [G/H] \big) = 
\dim_{\F_p} H_2 \big( H, \F_p \big) = 1
\end{equation}
where the last equality comes from  the fact that $H$ is a Demushkin group.

If $M$ is a free $\F_p \llb G \rrb$-module then clearly $\beta_1^G(M) = 0$, so let us assume that this is not the case.
Equation \eqref{OneRelatorEq} then tells us that $\dim_{\F_p} H_1(G,M)= 1$,
and from \propref{HFPprop} (using the finite generation of $U$) we infer that $M$ is finitely generated.
Therefore, $M$ is a one-relator $G$-module, so $\beta_1^G(M) = 0$ by \corref{NoBeta1For1Relator}.
\end{proof}

The purpose of the following proposition is to show that upon passing to an open subgroup $H$ of $G$, one
can find an $H$-submodule of $\F_p \llb G/U \rrb$ with a vanishing relation gradient, and effectively bounded codimension.
For our inductive argument to work, we use a slightly more general formulation, leaving the case that is of interest for us to the corollary that follows.

\begin{prop} \label{EffProp}

Let $G$ be a nonsolvable pro-$p$ Demushkin group, and let $M$ be a finitely presented $\F_p \llb G \rrb$-module.
Suppose that $M$ has an open $G$-submodule $M_0$ with $\beta_1^G(M_0) = 0$, and let $H$ be an open subgroup of $G$ that acts trivially on $M/M_0$.
Then there exists an $H$-submodule $M'$ of $M$ with
\begin{equation}
\dim_{\F_p} M/M' \leq \beta_1^G(M), \quad \beta_1^H(M') = 0, \quad M_0 \subseteq M'.
\end{equation}

\end{prop}

\begin{proof}

We induct on the codimension of $M_0$ in $M$, and in the base case $M_0 = M$ we take $M' = M_0$. 
By index-proportionality, as established in \corref{IPCor}, we get that 
\begin{equation}
\beta_1^H(M') = \beta_1^H(M_0) = [G : H]\beta_1^G(M_0) = 0.
\end{equation}

Assume $M_0 \lneq M$, and let $M_1$ be a codimension one $G$-submodule of $M$ that contains $M_0$.
By \propref{HFPprop}, $M_1$ is a finitely presented $G$-module.
We can thus use induction to find an $H$-submodule $M_1'$ of $M_1$ such that
\begin{equation} \label{IndAssumEq}
\dim_{\F_p} M_1/M_1' \leq \beta_1^G(M_1), \quad \beta_1^H(M_1') = 0, \quad M_0 \subseteq M_1'.
\end{equation}
By monotonicity of the relation gradient, as given in \propref{MonProp}, we have $\beta_1^G(M_1) \leq \beta_1^G(M)$.
If this inequality is strict, integrality implies that
\begin{equation}
\beta_1^G(M) \geq \beta_1^G(M_1) + 1
\end{equation}
so we can take $M' = M_1'$. 
We can therefore assume that 
\begin{equation} \label{EqAssEq}
\beta_1^G(M) = \beta_1^G(M_1).
\end{equation}

Recall that $H$ acts trivially on $M/M_1'$, so by picking $a \in M \setminus M_1$ and taking $M'$ to be the $H$-submodule of $M$ generated by $M_1'$ and $a$, we see that
\begin{equation} \label{IntSumEq}
M' \cap M_1 = M_1', \quad M' + M_1 = M.
\end{equation}
It follows that the short exact sequence
\begin{equation}
0 \to M' \cap M_1 \to M' \oplus M_1 \to M' + M_1 \to 0
\end{equation}
of finitely presented $H$-modules, can be rewritten as
\begin{equation}
0 \to M_1' \to M' \oplus M_1 \to M \to 0.
\end{equation}
Using the subadditivity of the relation gradient in short exact sequences, obtained in \propref{SaProp}, 
and recalling equation \eqref{IndAssumEq}, we find that
\begin{equation} \label{ALCEq}
\beta_1^H(M') + \beta_1^H(M_1) = \beta_1^H(M' \oplus M_1) \leq \beta_1^H(M_1') + \beta_1^H(M) = \beta_1^H(M).
\end{equation}
Applying the index-proportionality of \corref{IPCor} to equation \eqref{EqAssEq} gives
\begin{equation}
\beta_1^H(M) = \beta_1^H(M_1).
\end{equation}
Plugging this into inequality \eqref{ALCEq} shows that
\begin{equation}
\beta_1^H(M') + \beta_1^H(M_1) \leq \beta_1^H(M_1)
\end{equation}
so $\beta_1^H(M') = 0$.
At last, from equations \eqref{IntSumEq}, \eqref{IndAssumEq}, \eqref{EqAssEq} we get that
\begin{equation}
\dim_{\F_p} M/M' = \dim_{\F_p} M_1/M_1' \leq \beta_1^G(M_1) = \beta_1^G(M) 
\end{equation}
completing the induction and the proof.
\end{proof}

\begin{cor} \label{VVCor}

Let $G$ be a nonsolvable pro-$p$ Demushkin group, and let $U$ be a finitely generated subgroup of $G$.
Then there exists an open subgroup $H$ of $G$ and an $H$-submodule $N$ of the $G$-module $\F_p \llb G/U \rrb$ with 
\begin{equation}
\dim_{\F_p} \F_p \llb G/U \rrb \big/ N \leq \beta_1^G \big( \F_p \llb G/U \rrb \big), \quad \beta_1^H(N) = 0.
\end{equation}

\end{cor}

\begin{proof}

\propref{IneffProp} provides us with an open $G$-submodule $M_0$ of $\F_p \llb G/U \rrb$ with $\beta_1^G(M_0) = 0$.
Set $V \defeq \F_p \llb G/U \rrb \big/ M_0$ and let
\begin{equation}
\rho \colon G \to \mathrm{GL}(V)
\end{equation}
be the homomorphism associated to the $G$-module structure on $V$. Put
\begin{equation}
H \defeq \mathrm{Ker}(\rho)
\end{equation}
and note that $H$ is an open subgroup of $G$ that acts trivially on $V$.
Now just invoke \propref{EffProp} with $M = \F_p \llb G/U \rrb$ and take $N = M'$.
\end{proof}

\section{Submultiplicativity}

By \corref{Reason}, the strengthened Hanna Neumann conjecture (as stated in \thmref{FirstRes}) is tantamount to the submultiplicativity 
\begin{equation} \label{SMEq}
\beta_1^G \big(  \F_p \llb G/U \rrb \ \widehat\otimes_{\F_p} \ \F_p \llb G/W \rrb \big ) \leq 
\beta_1^G \big( \F_p \llb G/U \rrb \big )\beta_1^G \big( \F_p \llb G/W \rrb \big )
\end{equation}
of the relation gradient (recall that $[G:U], [G : W]$ are infinite). This inequality is established herein.

\begin{proof}

\corref{VVCor} provides us with an $H \leq_o G$ and an $H$-submodule $N$ of $\F_p \llb G/U \rrb$ such that
\begin{equation} \label{FromPrevCorEq}
\dim_{\F_p} \F_p \llb G/U \rrb \big/ N \leq \beta_1^G \big( \F_p \llb G/U \rrb \big), \quad \beta_1^H(N) = 0.
\end{equation}
By the index-proportionality of the relation gradient from \corref{IPCor},
inequality \eqref{SMEq} is readily equivalent to the inequality
\begin{equation} \label{SMREq}
\beta_1^H \big(  \F_p \llb G/U \rrb \ \widehat\otimes_{\F_p} \ \F_p \llb G/W \rrb \big ) \leq 
\beta_1^G \big( \F_p \llb G/U \rrb \big )\beta_1^H \big( \F_p \llb G/W \rrb \big ).
\end{equation}
The subadditivity of the relation gradient in exact sequences, established in \propref{SaProp}, allows us to bound the left hand side of equation \eqref{SMREq} by
\begin{equation} \label{GWNEq}
\beta_1^H \big(  N \ \widehat\otimes_{\F_p} \ \F_p \llb G/W \rrb \big ) + 
\beta_1^H \big(  \F_p \llb G/U \rrb \big/ N \ \widehat\otimes_{\F_p} \ \F_p \llb G/W \rrb \big ).
\end{equation}

Consider the second summand in equation \eqref{GWNEq}.
By equation \eqref{FromPrevCorEq}, the $H$-module $\F_p \llb G/U \rrb \big/ N$ has a filtration of length at most $\beta_1^G \big( \F_p \llb G/U \rrb \big)$ with one-dimensional consecutive quotients.
Consequently, the $H$-module
\begin{equation} 
\F_p \llb G/U \rrb \big/ N \ \widehat\otimes_{\F_p} \ \F_p \llb G/W \rrb
\end{equation}
has a filtration of length at most $\beta_1^G \big( \F_p \llb G/U \rrb \big)$ 
with quotients isomorphic to $\F_p \llb G/W \rrb$. 
Hence, an inductive application of \propref{SaProp} gives
\begin{equation}
\beta_1^H \big(  \F_p \llb G/U \rrb \big/ N \ \widehat\otimes_{\F_p} \ \F_p \llb G/W \rrb \big ) \leq
\beta_1^G \big( \F_p \llb G/U \rrb \big) \beta_1^H \big( \F_p \llb G/W \rrb \big)
\end{equation}
which coincides with the right hand side of \eqref{SMREq}.
Thus, in order to prove inequality \eqref{SMREq} it suffices to show that the first summand in \eqref{GWNEq} vanishes.

By \corref{HomRefCor}, the $G$-module $\F_p \llb G/W \rrb$ is finitely related over $U$,
so by \lemref{VirtuallyFreeLem} there exists an open $G$-submodule $M$ of $\F_p \llb G/W \rrb$ with
\begin{equation} \label{b1um0Eq}
\beta_1^U(M) = 0.
\end{equation}
By the subadditivity of the relation gradient from \propref{SaProp}, we have
\begin{equation}
\beta_1^H \big(  N \ \widehat\otimes_{\F_p} \ \F_p \llb G/W \rrb \big ) \leq
\beta_1^H \big(  N \ \widehat\otimes_{\F_p} \ M \big ) +
\beta_1^H \big(  N \ \widehat\otimes_{\F_p} \ \F_p \llb G/W \rrb \big/ M \big ).
\end{equation}
For the second summand above, 
upon repeating the filtration argument from the preceding paragraph, 
we conclude from \propref{SaProp} that 
\begin{equation}
\beta_1^H \big(  N \ \widehat\otimes_{\F_p} \ \F_p \llb G/W \rrb \big/ M \big ) \leq
\dim_{\F_p} \F_p \llb G/W \rrb \big/ M \cdot \beta_1^H \big(  N \big )
\end{equation}
and the right hand side vanishes in view of equation \eqref{FromPrevCorEq}.
Hence, we are left with the task of showing that $\beta_1^H (  N \ \widehat\otimes_{\F_p} \ M)$ vanishes.

By \eqref{FromPrevCorEq} the $H$-module $\F_p \llb G/U  \rrb \big/ N \ \widehat\otimes_{\F_p} \ M$ admits a (finite) filtration with consecutive quotients isomorphic to $M$.
By \corref{MonH2Cor} we have
\begin{equation}
\dim_{\F_p} H_2 \big( H, M \big) \leq \dim_{\F_p} H_2 \big( H, \F_p \llb G/W \rrb \big)
\end{equation}
and this vanishes in view of equation \eqref{Eq2}.
We conclude that
\begin{equation}
H_2 \big( H, \F_p \llb G/U  \rrb \big/ N \ \widehat\otimes_{\F_p} \ M \big) = 0.
\end{equation}
We can thus invoke \propref{MonProp} to get that
\begin{equation}
\beta_1^H \big(  N \ \widehat\otimes_{\F_p} \ M \big) \leq 
\beta_1^H \big( \F_p \llb G/U \rrb \ \widehat\otimes_{\F_p} \ M \big)
\end{equation}
so by index-proportionality from \corref{IPCor} it suffices to show that
\begin{equation}
\beta_1^G \big( \F_p \llb G/U \rrb \ \widehat\otimes_{\F_p} \ M \big) = 0.
\end{equation}
This follows from equation \eqref{b1um0Eq} and Shapiro's lemma (for the relation gradient) as obtained in \propref{b1Shap}.
\end{proof}

\section*{Acknowledgments}
This paper is partially supported by the Spanish MINECO through the grants MTM2014-53810-C2-01,  MTM2017-82690-P and the ``Severo Ochoa" program for Centres of Excellence  (SEV-2015-0554).
Mark Shusterman is grateful to the Azrieli Foundation for the award of an Azrieli Fellowship.
The second author was partially supported by a grant of the Israel Science Foundation with cooperation of UGC no. 40/14.
We thank the referee whose remarks and suggestions substantially improved the article.

\end{document}